\documentclass[12pt]{amsart}
\usepackage{amsmath,amssymb,amsbsy,amsfonts,amsthm,latexsym,mathabx,
            amsopn,amstext,amsxtra,euscript,amscd,stmaryrd,mathrsfs,
            cite,array,mathtools}
            
\usepackage{url}
\usepackage[colorlinks,linkcolor=blue,anchorcolor=blue,citecolor=blue]{hyperref}
\usepackage{color}
\begin{document}

\newtheorem{theorem}{Theorem}
\newtheorem{lemma}[theorem]{Lemma}
\newtheorem{claim}[theorem]{Claim}
\newtheorem{cor}[theorem]{Corollary}
\newtheorem{prop}[theorem]{Proposition}
\newtheorem{definition}[theorem]{Definition}
\newtheorem{question}[theorem]{Question}
\newcommand{\hh}{{{\mathrm h}}}

\numberwithin{equation}{section}
\numberwithin{theorem}{section}
%\numberwithin{definition}{section}

\def\sssum{\mathop{\sum\!\sum\!\sum}}
\def\ssum{\mathop{\sum\ldots \sum}}

\def \balpha{\boldsymbol\alpha}
\def \bbeta{\boldsymbol\beta}
\def \bgamma{{\boldsymbol\gamma}}
\def \bomega{\boldsymbol\omega}

\def\sssum{\mathop{\sum\!\sum\!\sum}}
\def\ssum{\mathop{\sum\ldots \sum}}
\def\dsum{\mathop{\sum\  \sum}}
\def\iint{\mathop{\int\ldots \int}}

\def\squareforqed{\hbox{\rlap{$\sqcap$}$\sqcup$}}
\def\qed{\ifmmode\squareforqed\else{\unskip\nobreak\hfil
\penalty50\hskip1em\null\nobreak\hfil\squareforqed
\parfillskip=0pt\finalhyphendemerits=0\endgraf}\fi}%%

%  use the AMS-Euler Fraktur fonts
%%%%%%%%%%%%%%%%%%%%%%%%%%%%%%%%%%
\newfont{\teneufm}{eufm10}
\newfont{\seveneufm}{eufm7}
\newfont{\fiveeufm}{eufm5}
%%%%%%%%%%%%%%%%%%%%%%%%%%%%%%%%%
%
%  allow automatic size selection in math mode
%
%%%%%%%%%%%%%%%%%%%%%%%%%%%%%%%%%
\newfam\eufmfam
     \textfont\eufmfam=\teneufm
\scriptfont\eufmfam=\seveneufm
     \scriptscriptfont\eufmfam=\fiveeufm
%%%%%%%%%%%%%%%%%%%%%%%%%%%%%%%%%
%
%  \frak works on a single symbol at a time...
%
\def\frak#1{{\fam\eufmfam\relax#1}}

\def\fK{\mathfrak K}
\def\fT{\mathfrak{T}}

\def\fA{{\mathfrak A}}
\def\fB{{\mathfrak B}}
\def\fC{{\mathfrak C}}
\def\fD{{\mathfrak D}}

\newcommand{\sX}{\ensuremath{\mathscr{X}}}

\def\eqref#1{(\ref{#1})}

\def\vec#1{\mathbf{#1}}
\def\dist{\mathrm{dist}}
\def\vol#1{\mathrm{vol}\,{#1}}

\def\squareforqed{\hbox{\rlap{$\sqcap$}$\sqcup$}}
\def\qed{\ifmmode\squareforqed\else{\unskip\nobreak\hfil
\penalty50\hskip1em\null\nobreak\hfil\squareforqed
\parfillskip=0pt\finalhyphendemerits=0\endgraf}\fi}

\def\sA{\mathscr A}
\def\sB{\mathscr B}
\def\sC{\mathscr C}
\def\sD{\Delta}
\def\sE{\mathscr E}
\def\sF{\mathscr F}
\def\sG{\mathscr G}
\def\sH{\mathscr H}
\def\sI{\mathscr I}
\def\sJ{\mathscr J}
\def\sK{\mathscr K}
\def\sL{\mathscr L}
\def\sM{\mathscr M}
\def\sN{\mathscr N}
\def\sO{\mathscr O}
\def\sP{\mathscr P}
\def\sQ{\mathscr Q}
\def\sR{\mathscr R}
\def\sS{\mathscr S}
\def\sU{\mathscr U}
\def\sT{\mathscr T}
\def\sV{\mathscr V}
\def\sW{\mathscr W}
\def\sX{\mathscr X}
\def\sY{\mathscr Y}
\def\sZ{\mathscr Z}

%%%%%%%%%%%%%%%%%%%%%%%%%
% Alphabet calligraphie %
%%%%%%%%%%%%%%%%%%%%%%%%%
\def\cA{{\mathcal A}}
\def\cB{{\mathcal B}}
\def\cC{{\mathcal C}}
\def\cD{{\mathcal D}}
\def\cE{{\mathcal E}}
\def\cF{{\mathcal F}}
\def\cG{{\mathcal G}}
\def\cH{{\mathcal H}}
\def\cI{{\mathcal I}}
\def\cJ{{\mathcal J}}
\def\cK{{\mathcal K}}
\def\cL{{\mathcal L}}
\def\cM{{\mathcal M}}
\def\cN{{\mathcal N}}
\def\cO{{\mathcal O}}
\def\cP{{\mathcal P}}
\def\cQ{{\mathcal Q}}
\def\cR{{\mathcal R}}
\def\cS{{\mathcal S}}
\def\cT{{\mathcal T}}
\def\cU{{\mathcal U}}
\def\cV{{\mathcal V}}
\def\cW{{\mathcal W}}
\def\cX{{\mathcal X}}
\def\cY{{\mathcal Y}}
\def\cZ{{\mathcal Z}}
\newcommand{\rmod}[1]{\: \mbox{mod} \: #1}

\def\vr{\mathbf r}

\def\e{{\mathbf{\,e}}}
\def\ep{{\mathbf{\,e}}_p}
\def\em{{\mathbf{\,e}}_m}
\def\en{{\mathbf{\,e}}_n}

\def\Tr{{\mathrm{Tr}}}
\def\Nm{{\mathrm{Nm}}}

 \def\SS{{\mathbf{S}}}

\def\lcm{{\mathrm{lcm}}}

\def\({\left(}
\def\){\right)}
\def\fl#1{\left\lfloor#1\right\rfloor}
\def\rf#1{\left\lceil#1\right\rceil}

\def\mand{\qquad \mbox{and} \qquad}

\newcommand{\commG}[1]{\marginpar{%
\begin{color}{red}
\vskip-\baselineskip %raise the marginpar a bit
\raggedright\footnotesize
\itshape\hrule \smallskip G: #1\par\smallskip\hrule\end{color}}}

\newcommand{\commI}[1]{\marginpar{%
\begin{color}{blue}
\vskip-\baselineskip %raise the marginpar a bit
\raggedright\footnotesize
\itshape\hrule \smallskip I: #1\par\smallskip\hrule\end{color}}}

\newcommand{\commII}[1]{\marginpar{%
\begin{color}{magenta}
\vskip-\baselineskip %raise the marginpar a bit
\raggedright\footnotesize
\itshape\hrule \smallskip I: #1\par\smallskip\hrule\end{color}}}

%%%%%%%%%%%%%%%%%%%%%%%%%%%%%%%%%%%%%%%%%%%%%%%%%%%%%%%%
%%%%%%%%%%%%%%%%%%%%%%%%%%%%%%%%%%%%%%%%%%%%%%%%%%%%%%%%
%%%%%%%%%%%%%%%%%%%%%%%%%%%%%%%%%%%%%%%%%%%%%%%%%%%%%%%%
%%%%%%%%%%%%%%%%%%%%%%%%%%%%%%%%%%%%%%%%%%%%%%%%%%%%%%%%

%%%%%%%  END OF STANDARD STUFF %%%%%%%%%

%%%%%%%%%%%%%%%%%%%%%%%%%%%%%%%%%%%%%%%%%%%%%%%%%%%%%%%%
%%%%%%%%%%%%%%%%%%%%%%%%%%%%%%%%%%%%%%%%%%%%%%%%%%%%%%%%
%%%%%%%%%%%%%%%%%%%%%%%%%%%%%%%%%%%%%%%%%%%%%%%%%%%%%%%%
%%%%%%%%%%%%%%%%%%%%%%%%%%%%%%%%%%%%%%%%%%%%%%%%%%%%%%%
%%%%%%%%%%%
%%% Spell

\hyphenation{re-pub-lished}

\parskip 4pt plus 2pt minus 2pt
%% \parskip= 2 pt plus 3pt

%% \parindent 0 pt

%\mathsurround=1pt

\def\bfdefault{b}
\overfullrule=5pt

\def \F{{\mathbb F}}
\def \K{{\mathbb K}}
\def \Z{{\mathbb Z}}
\def \Q{{\mathbb Q}}
\def \R{{\mathbb R}}
\def \C{{\\mathbb C}}
\def\Fp{\F_p}
\def \fp{\Fp^*}

\title[Trilinear  and Quadrilinear Exponential Sums]{Bounds of Trilinear and Quadrilinear Exponential Sums}

 \author[G. Petridis] {Giorgis Petridis}

\address{Department of Mathematics, University of Georgia, 
Athens, GA 30602, USA}
\email{giorgis@cantab.net}

 \author[I. E. Shparlinski] {Igor E. Shparlinski}

\address{Department of Pure Mathematics, University of New South Wales,
Sydney, NSW 2052, Australia}
\email{igor.shparlinski@unsw.edu.au}

\begin{abstract} We use an estimate  of Aksoy Yazici, Murphy, Rudnev and 
Shkredov (2016)  on the number of solutions of certain equations involving products and differences of sets in prime finite fields 
to give an explicit upper bound on trilinear exponential sums which improves the previous bound of Bourgain 
and Garaev (2009). We also obtain explicit bounds for quadrilinear exponential sums. 
\end{abstract}


\maketitle

\section{Introduction}

\subsection{Background} 

Let $p$ be a prime and let $\F_p$ be the finite field of $p$ elements.
Now given three sets $\cX, \cY, \cZ \subseteq \F_p$, and three sequences of  weights
 $\alpha= (\alpha_{x})_{x\in \cX}$, $\beta = \( \beta_{y}\)_{y \in \cY}$ and $\gamma =    \(\gamma_{z}\)_{z \in \cZ}$
 supported on
$\cX$,  $\cY$ and 
$\cZ$, respectively, we consider exponential sums
\begin{equation}
\label{eq:Sum S}
S(\cX, \cY, \cZ;\alpha, \beta, \gamma) = \sum_{x \in\cX} \sum_{y \in \cY}
 \sum_{z\in \cZ}\alpha_{x} \beta_{y} \gamma_{z}\ep(xyz) ,  
\end{equation}
where $\ep(z) = \exp(2 \pi i z/p)$.  
We recall that the bilinear analogues of  these sums are classical and have been studied in several papers, in which case  
for any sets $\cX, \cY \subseteq \F_p$ and any  $\alpha= (\alpha_{x})_{x\in \cX}$, $\beta = \( \beta_{y}\)_{y \in \cY}$, 
with 
$$
\sum_{x\in \cX}|\alpha_{x}|^2 = A \mand  \sum_{y \in \cY}|\beta_{y}|^2 = B, 
$$
we have 
\begin{equation}
\label{eq:bilin}
\left |\sum_{x \in \cX}\sum_{y \in \cY} \alpha_{x} \beta_{y}  \ep(xy) \right| \le \sqrt{pAB}, 
\end{equation}
see, for example,~\cite[Equation~1.4]{BouGar} or~\cite[Lemma~4.1]{Gar2}.

The trilinear  sums~\eqref{eq:Sum S} have been introduced and  estimated by Bourgain and Garaev~\cite{BouGar}. 
In particular, for and sets $\cX, \cY, \cZ \subseteq \F_p$ of cardinalities 
$$
\# \cX = X, \qquad \# \cY = Y, \qquad \# \cZ = Z,
$$
and  weights with 
$$
\max_{x \in \cX} |\alpha_{x}| \le 1, \qquad 
\max_{y \in   \cY}  |\beta_{y}| \le 1, \qquad 
\max_{z \in   \cZ}   |\gamma_{z}| \le 1,
$$ 
by~\cite[Theorem~1.2]{BouGar} we have
\begin{equation}
\label{eq:BG}
\left| S(\cX, \cY, \cZ;\alpha, \beta,\gamma)\right| \le (XYZ)^{13/16}p^{5/18 + o(1)}, 
\end{equation}
as $p\to \infty$. 
The bound~\eqref{eq:BG} has been generalised and extended in various directions, 
see~\cite{Bou1, Bou2, BouGlib,Gar2, Ost}. In particular, Bourgain~\cite{Bou1} has obtained 
a bound on multilinear sums
\begin{equation}
\label{eq:Bourg}
\left|\sum_{x_1 \in\cX_1}  \ldots  \sum_{x_n\in \cX_n}  \ep(x_1 \ldots x_n)\right| \le  X_1 \ldots X_n p^{-\eta}
\end{equation}
under an optimal condition 
$$
\min_{i=1, \ldots, n} X_i \ge  p^\delta \mand \prod_{i=1}^n X_i  \ge p^{1+\delta}, 
$$
of the sizes $X_i = \cX_i$, $i=1, \ldots, n$,  of the sets involved, where $\eta > 0$ 
depends only on the arbitrary parameter $\delta > 0$.
The interest to  multilinear exponential sums partially comes from applications to
exponential sums over subgroups of small order over finite fields, which has been used in the celebrated work of  
 Bourgain, Glibichuk and Konyagin~\cite{BGK}. 
We also note that several consecutive applications of the Cauchy inequality allow us to 
reduce general multilinear sums to sums without weights, see, for example, Lemma~\ref{lem:ExpSum}
below. 
Furthermore, the bound~\eqref{eq:Bourg} has been extended to arbitrary finite fields 
by Bourgain and Glibichuk~\cite{BouGlib}, see also~\cite{Bou2, Ost}. 
However, prior to the present work, the bound~\eqref{eq:BG}  has been the best 
known explicit bound  on the sums~\eqref{eq:Sum S}. 

We also note that  Hegyv{\'a}ri~\cite[Theorem~3.1]{Heg} has given estimates of 
some multilinear sums, without weights, over sets
of special additive structure (for sets with small difference sets).

Here we use a different approach to estimating the sums~\eqref{eq:Sum S}
and a recent result of  Aksoy Yazici, Murphy, Rudnev and 
Shkredov~\cite{AYMRS}  to obtain a different bound, which, in particular, improves~\eqref{eq:BG} when $p \to \infty$.

Furthermore, using a slightly different approach, based on another result of  Aksoy Yazici, Murphy, Rudnev and 
Shkredov~\cite{AYMRS} (see also~\cite{Pet}) we consider sums 
with more general weights to which the above approach does not apply. We also consider their multivariate versions.
In particular, given  $n$ sets $\cX_i \subseteq \F_p$, 
and also $n$ sequences of weights
$\omega_i= (\omega_i(\vec{x}))_{\vec{x} \in \F_p^n}$ such that $\omega_i(\vec{x})$ does not 
depend on the $i$th coordinate of the vector $\vec{x} = (x_1, \ldots, x_n) \in \F_p^n$, $i=1, \ldots,n$, 
we consider multilinear exponential sums
\begin{equation}
\begin{split}
\label{eq:Sum T}
T(\cX_1, \ldots, \cX_n;  \omega_1&, \ldots, \omega_n)\\
& = \ssum_{\vec{x} \in\cX_1\times \ldots \times \cX_n} 
\omega_1(\vec{x}) \ldots \omega_n(\vec{x}) \ep(x_1 \ldots x_n).
\end{split}
\end{equation}
As we have mentioned,  if the strength of the bound is not of concern but only the range 
of non-triviality is  important, then Bourgain~\cite{Bou1}  provides an optimal 
result~\eqref{eq:Bourg} for sums with constant weights, which can be extended to 
sums~\eqref{eq:Sum T} with several consecutive applications of the Cauchy inequality
as in Lemma~\ref{lem:ExpSum} below.  However the saving in such a bound is rather 
small, while here we are interested in stronger and more explicit bounds.

Although we estimate the sums~\eqref{eq:Sum T} only for $n=3$ and $n=4$ we develop some 
tools in Section~\ref{sec:ExpEnergy} in full generality, which may be of use if the results 
of Section~\ref{sec:AddComb} get eventually extended to equations with more sets and variables.

Our method is based on a upper bound of Rudnev~\cite{Rud} on the number of incidences between a set of points and a set of planes in $\F_p^3$. Rudnev's paper, which undoubtedly will find many more applications, stems from the  Guth and Katz~\cite{GuKat2} solution to the Erd\H{o}s distinct distance problem for planar sets and the Klein--Pl\"ucker line geometry formalism~\cite[Chapter~2]{PotWal}. So Rudnev's  work~\cite{Rud} indirectly depends on classical techniques such as the  polynomial method (see~\cite{Guth} and~\cite[Chapter~9]{TaoVu}) and properties of ruled surfaces (see~\cite{Katz}) and the Klein quadric (see~\cite{RudSel}). A more detailed discussion can be found in the beginning of 
Section~\ref{sec:AddComb}.

We also illustrate potential applications of our results on a example of a certain question from additive 
combinatorics complementing those of 
S{\'a}rk{\"o}zy~\cite{Sark} on nonlinear equations with variables from 
arbitrary sets  and of Aksoy Yazici, Murphy, Rudnev and 
Shkredov~\cite{AYMRS} on sizes of polynomial images. 
In fact both are closely related and also both can be approached via the idea 
of Garaev~\cite{Gar1} which links multilinear exponential sums,  equation 
with variables from arbitrary sets and sums-product type results (see
 also~\cite{Heg} for some other applications of this idea).  So it is not surprising 
that our results fit well into this approach, see Section~\ref{sec:appl}. More precisely, 
our bounds of  multilinear exponential sums, allow to derive results for multifold sums 
and products, which have recently become a subject of  quite active investigation. 
For example, to put our results of Section~\ref{sec:appl} in a proper context, we present 
one of the  bounds of Roche-Newton, Rudnev and  Shkredov~\cite{RNRS}. 
Namely, by~\cite[Corollary~12]{RNRS} we have 
$$
\(\#(\cA + \cA +\cA)\)^4 \(\#(\cA \cA)\)^9 \ge \(\#\cA \)^{16}
$$
for any set $\cA \subseteq \F_p$ with $\#\cA  = O\(p^{18/35}\)$
(we refer to~\eqref{eq:sumdiffprod} for the definition sum and 
product sets); see also~\eqref{eq:ABC Fq1},  \eqref{eq:ABC Fq2} and~\eqref{eq:2diff3} below. 

Finally, we recall that the bound~\eqref{eq:bilin} has  a full analogue for sums 
with a nontrivial multiplicative characters $\chi$ of $\F_p^*$ with two sets 
 $\cX, \cY \subseteq \F_p$.
Chang~\cite{Chang} has given a better bound if one of these sets has a small sum set.
 More recently, 
Hanson~\cite{Han}, also using methods of additive 
combinatorics, 
has obtained a series of results which apply to trilinear character sums.
Shkredov and Volostnov~\cite{ShkVol} have given further 
improvements  of the results of~\cite{Chang,Han}. Unfortunately our approach does 
not seem to apply to multiplicative character sums.

\subsection{Notation}

We always use the letter $p$ to denote a prime number and  use the letter $q$ to denote a  prime power. 

Before we formulate our results, we recall that the notations $U = O(V)$,  $U \ll V$  and $V \gg U$ are all
equivalent to the assertion that the inequality $|U|\le c|V|$ holds for some
constant $c>0$.  
We also write $U \asymp V$ if $U \ll V \ll U$. 
Throughout the paper, the implied constants in the symbols `$O$',  `$\ll$' 
and `$\gg$'  are absolute.

\subsection{New bounds of exponential sums} 

We note that to simplify the results and exposition we assume that $0$ is 
excluded from the sets under consideration.
This changes the absolute value of, say, the sum 
$S(\cX, \cY, \cZ ;  \alpha, \beta, \gamma)$   by at most  $O(XY + XZ + YZ)$. 

\begin{theorem}
\label{thm:Bound S3}
For any sets $\cX, \cY, \cZ \subseteq \F_p^*$ of cardinalities $X, Y, Z$, respectively, with 
$X \ge Y \ge Z$
and any weights 
$\alpha= (\alpha_{x})$,  $\beta = (\beta_{y})$ and $\gamma=(\gamma_{z})$
with
$$
\max_{x \in \cX} |\alpha_{x}| \le 1, \qquad 
\max_{y \in   \cY}  |\beta_{y}| \le 1, \qquad 
\max_{z \in   \cZ}   |\gamma_{z}| \le 1,
$$ 
we have
$$
S(\cX, \cY, \cZ ;  \alpha, \beta, \gamma)    \ll    
 p^{1/4} X^{3/4} Y^{3/4} Z^{7/8} . 
$$
\end{theorem}

Theorem~\ref{thm:Bound S3} in nontrivial and improves the bound 
\begin{equation}
\label{eq:trilin triv}
| S(\cX, \cY, \cZ ;  \alpha, \beta, \gamma) |  \le p^{1/2}X^{1/2} Y^{1/2} Z,
\end{equation}
which is instant from~\eqref{eq:bilin}, if 
$$
X^{1/4} Y^{1/4} Z^{1/8} \gg p^{1/4} \mand X^{1/4} Y^{1/4} Z^{-1/8} \ll p^{1/4}.
$$
We rewrite these inequalities as 
$$
pZ^{1/2} \gg X  Y \gg pZ^{-1/2}.
$$
In particular if $X = Y = Z $, then 
$$
S(\cX, \cY, \cZ ;  \alpha, \beta, \gamma) \ll p^{1/4} X^{19/8},
$$
which is nontrivial for $X \ge  p^{2/5}$. This range of non-triviality is inferior to that obtained by Bourgain~\cite{Bou1}.

It is easy to see that Theorem~\ref{thm:Bound S3} improves~\eqref{eq:BG} for all 
cardinalities $X,Y,Z$ when $p \to \infty$, since without loss of generality we can assume that 
$X \ge Y \ge Z$ and then we have
\begin{align*}
 p^{1/4} X^{3/4} Y^{3/4} Z^{7/8}
 & \le  p^{5/18} X^{3/4} Y^{3/4} Z^{7/8} \\
 & =  p^{5/18} (XYZ)^{13/16} (Z/(XY))^{1/16}\\
 & \le p^{5/18} (XYZ)^{13/16}.
\end{align*}

Also, to indicate the strength of Theorem~\ref{thm:Bound S3} we compare it with the  bound  of Hegyv{\'a}ri~\cite[Corollary~2]{Heg}.
First we note that the bound of multilinear sums from~\cite[Theorem~3.1]{Heg} makes use of only three 
sets and does not improve with the number of sets. Now for $X \asymp Y \asymp p^{1/2}$ and $Z \ge p^{3/8}$,
ignoring very restrictive additive conditions on the sets $\cX,$ and $\cY$, the bound of~\cite[Corollary~2]{Heg}
takes form $O\(p^{3/16} XYZ^{1/2}\) = O\(p^{19/16} Z^{1/2}\)$. Unlike Theorem~\ref{thm:Bound S3}, the set $\cZ$ need not be the smallest set. Once $\cZ$ is assumed to be the smallest set, then Theorem~\ref{thm:Bound S3} gives the bound $O\(p Z^{7/8}\)$, which is stronger in the range $Z < p^{1/2}$ (in which case $\cZ$  is the smallest set anyway).

We also  compare  Theorem~\ref{thm:Bound S3}  with the bound  of 
Garaev~\cite[Theorem~4.2]{Gar2}, 
which work trilinear sums over arbitrary sets (but without weights). More precisely, 
by~\cite[Theorem~4.2]{Gar2},
if $XY\gg p$ then
$$
 \left| \sum_{x \in\cX} \sum_{y \in \cY}
 \sum_{z\in \cZ} \ep(xyz)  \right| \le XYZ^{539/540 + o(1)},
 $$ 
 while Theorem~\ref{thm:Bound S3} applies to weighted sums and yields (under the same condition $XY\gg p$)  the bound of the form 
  $O\(p^{1/4} X^{3/4} Y^{3/4} Z^{7/8}\) = O\(XY Z^{7/8}\)$.

As yet another evidence of the  efficiency of our approach and the strength of  Theorem~\ref{thm:Bound S3} 
we note that one can easily recover the second bound of~\cite[Corollary~19]{RNRS}.

We now obtain a bound of quadrilinear  analogues of sums~\eqref{eq:Sum S}
$$
S(\cW, \cX, \cY, \cZ;\alpha, \beta, \gamma, \delta) =   \sum_{w \in\cW} \sum_{x \in\cX} \sum_{y \in \cY}
 \sum_{z\in \cZ}\alpha_{w} \beta_{x} \gamma_{y}\delta_{z}\ep(wxyz) 
$$
with four sets $\cW, \cX, \cY, \cZ \subseteq \F_p^*$, and  weights
 $\alpha= (\alpha_{w})_{w\in \cW}$, $\beta = \( \beta_{x}\)_{x \in \cX}$,  $\gamma =    \(\gamma_{y}\)_{y \in \cY}$
 and $ \delta=    \(\delta_{z}\)_{z\in \cZ}$
 supported on $\cW$, 
$\cX$,  $\cY$ and 
$\cZ$, respectively.

To simplify the exposition we now assume that all cardinalities are less than $p^{2/3}$.

\begin{theorem}
\label{thm:Bound S4}
For any sets $\cW, \cX, \cY, \cZ \subseteq \F_p^*$ of cardinalities $W, X, Y, Z$, respectively, with 
$$
p^{2/3} \ge W \ge X \ge Y \ge Z ,
$$
and any weights 
$\alpha= (\alpha_{x})$,  $\beta = (\beta_{y})$,  $\gamma=(\gamma_{z})$  and $ \delta=    \(\delta_{z}\)$ 
with
$$
\max_{w \in \cW} |\alpha_{w}| \le 1, \qquad 
\max_{x \in \cX} |\beta_{x}| \le 1, \qquad 
\max_{y \in   \cY}  |\gamma_{y}| \le 1, \qquad 
\max_{z \in   \cZ}   |\delta_{z}| \le 1,
$$ 
we have
$$
S(\cW, \cX, \cY, \cZ;\alpha, \beta, \gamma, \delta)  \ll    
p^{1/8} W^{7/8} X^{7/8} Y^{15/16}  Z^{15/16}  . 
$$
\end{theorem}

If $W = X=Y=Z$, the bound in Theorem~\ref{thm:Bound S4} becomes
$$
S(\cW, \cX, \cY, \cZ;\alpha, \beta, \gamma, \delta)  \ll    
p^{1/8}  W^{29/8} ,
$$
which is nontrivial for $W \ge p^{1/3}$. Once again, the range of non-triviality is 
inferior to that obtained by Bourgain~\cite{Bou1}.

Next we move to the case of sums~\eqref{eq:Sum T} with more complicated weights.

\begin{theorem}
\label{thm:Bound T3} 
For any sets $\cX, \cY, \cZ \subseteq \F_p^*$ of cardinalities $X, Y, Z$, respectively, with 
$X \ge Y \ge Z$  
and weights 
$\rho= (\rho_{x,y})$,  $\sigma = (\sigma_{x,z})$ and $\tau=(\tau_{y,z})$
with 
$$
\max_{(x,y) \in \cX \times \cY} |\rho_{x,y}| \le 1, \quad 
\max_{(x,z) \in \cX \times \cZ} |\sigma_{x,z}| \le 1, \quad 
\max_{(y,z) \in \cY \times \cZ} |\tau_{y,z}| \le 1,  
$$
we have
$$
T(\cX, \cY, \cZ ;  \rho, \sigma, \tau) \ll  p^{1/8} X^{7/8} Y^{29/32}  Z^{29/32}  .
$$
\end{theorem}

For $X = Y = Z$,  Theorem~\ref{thm:Bound T3} is nontrivial in the same range 
 $X \ge  p^{2/5}$ like Theorem~\ref{thm:Bound S3}.

We also present an explicit bound for multilinear sums with four sets. 
Again, we make a simplifying  assumption that all cardinalities are less than $p^{2/3}$.

\begin{theorem}
\label{thm:Bound T4} 
For any sets $\cW, \cX, \cY, \cZ \subseteq \F_p^*$ of cardinalities $W, X, Y, Z$, respectively, with 
$$
p^{2/3} \ge W \ge X \ge Y \ge Z , 
$$
and weights $\vartheta = (\vartheta_{w,x,y})$, 
$\rho= (\rho_{w,x,z})$,  $\sigma = (\sigma_{w,y,z})$ and $\tau=(\tau_{x,y,z})$
with 
\begin{align*}
\max_{(w,x,y) \in \cW\times  \cX \times \cY} |\vartheta_{w,x,y}| \le 1, &\quad 
\max_{(w,x,z) \in \cW\times \cX \times \cZ} |\rho_{w,x,z}| \le 1, \\
\max_{(w,y,z) \in \cW\times  \cY \times \cZ} |\sigma_{w,y,z}| \le 1, &\quad 
\max_{(x,y,z) \in  \cX\times \cY \times \cZ} |\tau_{x,y,z}| \le 1,  
\end{align*}
we have
$$
T(\cW, \cX, \cY, \cZ ; \vartheta,  \rho, \sigma, \tau)   \ll   
 p^{1/16}  W^{15/16} (XY)^{61/64}   Z^{31/32}.
 $$
\end{theorem}

For $W =X  = Y =Z$,   the bound of Theorem~\ref{thm:Bound T4} becomes 
$$
T(\cW, \cX, \cY, \cZ ; \vartheta,  \rho, \sigma, \tau)   \ll   
 p^{1/16}  W^{61/16},
  $$
which is nontrivial for $W \ge  p^{1/3}$ which is the same range as for 
Theorem~\ref{thm:Bound T4}  in the case of sets of equal cardinalities.

Note that although the bound of  Theorems~\ref{thm:Bound T3} and~\ref{thm:Bound T4} are 
weaker than that of Theorems~\ref{thm:Bound S3} and~\ref{thm:Bound S4}, they however 
apply to more general sums, including, for example, to sums of the form 
\begin{equation}
\label{eq:cube quart}
\sum_{x \in\cX} \sum_{y \in \cY}  \sum_{z\in \cZ}\ \ep(F(x,y,z)), \quad 
\sum_{w \in\cW} \sum_{x \in\cX} \sum_{y \in \cY}  \sum_{z\in \cZ}\ \ep(G(w,x,y,z)), 
\end{equation}
for any cubic  polynomial $F(x,y,z) \in \F_p[x,y,z]$ and 
quartic  polynomial 
%%PROOFS  
% $G(w,x,y,z) \in \F_p[x,y,z]$  
    $G(w,x,y,z) \in \F_p[w,x,y,z]$  
that contain a term of the form $axyz$ and $awxyz$, 
respectively, with $a \ne 0$.

\subsection{Applications} 
\label{sec:appl}
Given two set $\cA,\cB \subseteq \F_q$ we define the sum,  difference and product sets
\begin{equation}
\label{eq:sumdiffprod}
\begin{split}
& \cA + \cB = \{a+b~:~a\in \cA, \ b \in \cB\}, \\ & \cA  - \cB   = \{a-b~:~a\in \cA, \ b \in \cB\},\\ 
& \cA \cB= \{ab~:~a\in \cA, \ b \in \cB\}.
\end{split}
\end{equation}
 These notations naturally extend to operations with any number of sets.
 
 First we recall that by a result~S{\'a}rk{\"o}zy~\cite{Sark}, for any sets $\cA,\cB,\cC  \subseteq \F_q^*$ 
 of cardinalities $A,B,C$, we have
\begin{equation}
\label{eq:ABC Fq1}
\#\(\cA\cB + \cC\) = q + O\(\frac{q^3}{ABC}\), 
 \end{equation}
see also~\cite[Equation~(9)]{Shp}. This immediately implies that there is an absolute constant $c_0$ 
such that  for any sets $\cA,\cB,\cC, \cD  \subseteq \F_q^*$ 
of cardinalities $A,B,C, D$  with 
$
ABCD \ge c_0q^3$, 
we have 
\begin{equation}
\label{eq:ABCD Fq}
 \cA\cB + \cC + \cD  = \F_q.
 \end{equation}
 Indeed, if there is $\lambda \in  \F_q \setminus(\cA\cB + \cC + \cD)$ then 
 $(\cA\cB + \cC) \cap(\lambda - \cD)  = \emptyset$ and thus 
 $$
D = \#(\lambda - \cD)< q - \#(\cA\cB + \cC)  =  O\(\frac{q^3}{ABC}\)
$$
by~\eqref{eq:ABC Fq1}.
Furthermore,  over a prime field $\F_p$,  Roche-Newton, Rudnev and  Shkredov~\cite[Theorem~1]{RNRS} give a 
lower bound 
\begin{equation}
\label{eq:ABC Fq2}
\#\(\cA\cB + \cC\)  \gg \min\{p, (ABC)^{1/2}, ABC/M\}
 \end{equation}
 where $M = \max\{A,B,C\}$.

We now consider the related question, involving triple products and four sets.

\begin{theorem}
\label{thm:Image1} 
For any sets $\cA,\cB,\cC, \cD \subseteq \F_p^*$ of cardinalities $A,B,C,D$, respectively, with 
$$
 A \ge B \ge C  , 
$$ 
we have
$$
\#\(\cA\cB \cC +\cD\) = p + O\( p^{5/2} A^{-1/2} B^{-1/2}  C^{-1/4}  D^{-1}\)
 $$
 and 
$$
\#\(\cA\cB \cC +\cD\)  \gg  \min\{p , p^{-1/2} A^{1/2} B^{1/2}  C^{1/4} D\}. 
 $$
\end{theorem}

In particular, we immediately derive a version of the property~\eqref{eq:ABCD Fq} for five sets
in $\F_p$.

\begin{cor}
\label{cor:ABCDE Fp1}  There is an absolute constant $c_0$ 
such that  for any sets $\cA,\cB,\cC, \cD, \cE  \subseteq \F_p^*$ 
of cardinalities $A,B,C, D, E$  with 
$$
A  B   C^{1/2}  D^2 E^2 \ge c_0 p^5,
$$ 
we have 
$$
 \cA\cB \cC + \cD +\cE  = \F_p. 
 $$
\end{cor}

Furthermore, we also define 
$$
\cA^k = \{a^k~:~a \in \cA\}
$$
(note that $\cA^k$ is not the $k$-fold product set of $\cA$ which sometimes is 
also denoted by $\cA^k$).

Aksoy Yazici, Murphy, Rudnev and 
Shkredov~\cite[Corollary~2.13(1)]{AYMRS} 
%%PROOFS - numbering of statement in printed version
have shown that 
for any set $\cA \subseteq \F_p$ of cardinality $A < p^{7/12}$ we have 
\begin{equation}
\label{eq:2diff3}
\#\( (\cA-\cA)^3 + (\cA-\cA)^3 \) \gg A^{36/35}. 
\end{equation}

 Here we obtain a related result  involving four sets from $\F_p$.

\begin{theorem}
\label{thm:Image2} 
For any sets $\cA,\cB,\cC, \cD \subseteq \F_p^*$ of cardinalities $A,B,C,D$, respectively, with 
$$
p^{2/3} \ge A \ge B \ge C ,  
$$
we have
$$
\#\((\cA+\cB+\cC)^3 + \cD\) = p + O\( p^{9/4} A^{-1/4} B^{-3/16}  C^{-3/16} D^{-1} \)
$$
and
$$
\#\((\cA+\cB+\cC)^3 + \cD\)  \gg  \min\{p , p^{-1/4} A^{1/4} B^{3/16}  C^{3/16} D\}.
$$
\end{theorem}

In particular, if $A=B=C=D$ the lower bounds of  Theorems~\ref{thm:Image1} and~\ref{thm:Image2} are 
nontrivial for $A\ge p^{2/5}$.   We also obtain yet another analogue of~\eqref{eq:ABCD Fq}:

\begin{cor}
\label{cor:ABCDE Fp2}  There is an absolute constant $c_0$ 
such that  for any sets $\cA,\cB,\cC, \cD, \cE  \subseteq \F_p^*$ 
of cardinalities $A,B,C, D, E$  with 
$$
A B^{3/4}  C^{3/4} D^{4} E^4 \ge c_0 p^9,
$$ 
we have 
$$
(\cA+\cB+\cC)^3 + \cD +\cE  = \F_p. 
 $$
\end{cor}

Theorems~\ref{thm:Image1} and~\ref{thm:Image2} are based on bounds of exponential sums 
of Theorems~\ref{thm:Bound S3} and~\ref{thm:Bound T3}, respectively. Using Theorems~\ref{thm:Bound S4} and~\ref{thm:Bound T4}
one can obtain versions of Theorems~\ref{thm:Image1} and~\ref{thm:Image2}
 for more complicated sets such as $\cA\cB \cC\cD +\cE$ and 
$(\cA+\cB+\cC+\cD)^4 + \cE$.

Our final application is an extension of the following inequality of Garaev~\cite[Theorem~1]{Gar1} to triple products 
\[
\#(\cA\cA) \#(\cA+\cA) \gg \min\left\{p A , \frac{A^4}{p} \right\}.
\]

\begin{theorem}
\label{thm:UV} 
Let $\cA, \cB, \cC, \cD\subseteq \F_p^*$ be sets of cardinalities $A$, $B$, $C$ and $D$, respectively.
Consider the sets 
$$
\cU = \cA \cB\cC \mand \cV = \cA +\cD 
$$
of cardinalities $U$ and  $V$, respectively.
Then we have  
$$
UV \gg   p A\qquad \text{or}\qquad
U^3 V^2  \gg A^4 B  C^{1/2} D^{2} p^{-1}.
$$
\end{theorem}

In particular Theorem~\ref{thm:UV} implies that  
$$
 \max\{U,V\}  \gg \min\{p^{1/2} A^{1/2},  A^{4/5} B^{1/5} C^{1/10} D^{2/5} p^{-1/5}\}.
$$
Furthermore, if $\cG\subseteq  \F_p^*$ is a multiplicative subgroup
of $\F_p^*$ of order $T$ then for any set $ \cS\subseteq \F_p^*$ with $\# \cS = S$,  by 
Theorem~\ref{thm:UV}  we have
\begin{equation}
\label{eq:Set+Group}
\#\( \cG +\cS \) \gg   \min\{p, S T^{5/4} p^{-1/2}\}, 
\end{equation}
which improves the trivial universal lower bound $S$ for $T \ge Cp^{2/5}$ with any  a sufficiently large constant $C$.

As before, we note that using Theorem~\ref{thm:Bound S4}  
one can obtain analogues of 
Theorem~\ref{thm:UV} 
 for more complicated sets. For example, Theorem~\ref{thm:Bound S4}   allows to 
deal with the sets  
 $$
\cU = \cA \cB\cC \cD \mand \cV = \cA +\cE 
$$
with $\cA, \cB, \cC, \cD, \cE \subseteq \F_p^*$. In turn one can obtain the following
version of~\eqref{eq:Set+Group}
$$
\#\( \cG +\cS \) \gg   \min\{p, S T^{3/4} p^{-1/4}\}, 
$$
which is now nontrivial  for $T \ge Cp^{1/3}$ with any  a sufficiently large constant $C$.

%%PROOFs - mention improvements in T(A)
 \subsection{Recent developement}
 \label{sec:StopThePress}

Our results depend on the forthcoming bounds on the quantity $T(\cU)$, c.f. Lemma~\ref{lem:TU}. The recent work~\cite{MPRNRS} contains an improvement when $p^{1/2} \leq U \leq p^{3/5}$, which leads to improved bounds for the exponential sum we consider in certain ranges of the cardinalities $X,U, V, W$.

\section{Preliminaries} 
 \subsection{Background from arithmetic combinatorics}
 \label{sec:AddComb}

Some of the results of this section apply to arbitrary fields $\F_q$ of $q$ elements 
so we formulate them in this form.

The proofs of Theorems~\ref{thm:Bound S3}--\ref{thm:Bound T4} come down to non-trivial upper bounds on the number of solutions to equations with variables in prescribed sets in $\F_p^*$. Of particular importance in our considerations is the following such quantity.

\begin{definition}\label{def:N}
Let $\cU, \cV, \cW \subseteq \F_q^*$. Then $N(\cU, \cV,\cW)$ denotes the number of solutions to 
\[
u_1(v_1-w_1) = u_2 (v_2-w_2)
\]
with $u_1, u_2 \in \cU$, $v_1,v_2 \in \cV$ and $w_1,w_2 \in \cW$.  
\end{definition}

A trivial upper bound for $N(\cU, \cV, \cW)$ in terms terms of the cardinalities $U = \# \cU, V = \# \cV, W = \# \cW$ is
\[
N(\cU, \cV, \cW) \le U V^2 W^2 + U^2 V W.
\]
This is because for each 5-tuple $(u_1, v_1,v_2, w_1,w_2)$ where $v_1 \neq w_1$ and $v_2 \neq w_2$ there is at most one $u_2$ satisfying $u_1(v_1-w_1) = u_2 (v_2-w_2)$; while if $v_1 = w_1$ there is no solution unless $v_2 = w_2$ and in this case all $u_2 \in \cU$ work. Our method for obtaining non-trivial upper bounds for the modulus of exponential sums relies on non-trivial upper bounds for $N(\cU, \cV, \cW)$. The first non-trivial upper bound (for the special case $U=V=W \le p^{2/3}$) can be traced back to the breakthrough paper of Bourgain, Katz and 
Tao~\cite{BKT} on sum-product questions in $\F_p$. Bourgain, Katz and Tao have shown~\cite[Theorem~6.2]{BKT} that, under mild conditions, a set of $n$ points in $\F_p^2$ and a set of $n$ lines in $\F_p^2$ determine $O(n^{3/2-c})$ points-lines incidences for an absolute albeit small $c>0$. From this a non-trivial upper bound for $N(\cU, \cV, \cW)$ follows easily, for example, by modifying~\cite[Lemma~3.5.1]{Dvir} in Dvir's survey~\cite{Dvir}. Progress over the years to the Bourgain-Katz-Tao result, see~\cite{BouGar,Gar0,HelRud,Jones,KatzShen,Li, LiRN,Rud0} and references therein, leads implicitly to improved upper bounds for $N(\cU, \cV, \cW)$.   

The most recent and   significant progress in estimating $N(\cU,\cV,\cW)$ has its 
roots in a bound of Rudnev~\cite{Rud} on the number of incidences between a set of points in $\F_p^3$ and a set of planes in $\F_p^3$. 
Rudnev's work~\cite{Rud} is based on a theorem of Guth and Katz~\cite[Theorem~2.10]{GuKat2} from their solution to the Erd\H{os} distinct distance conjecture for planar sets  and on the 19th century Pl\"ucker-Klein formalism for projective line geometry~\cite{PotWal}. Applying the incidence theorem of  Rudnev~\cite{Rud} to $N(\cU, \cV, \cW)$ requires an elegant trick and has been done by Aksoy Yazici, Murphy, Rudnev and Shkredov~\cite[Theorem~1]{AYMRS}. 
%%PROOFs - numbering of statement in printed version
Given the importance of Rudnev's points-planes incidence theorem to our results, it should be noted that the theorem is essentially sharp. Existing points-lines incidence results in $\F_p^3$ (see~\cite{EllHab,Kol} and also~\cite{GuKat1}) do not seem to work as well for bounding quantities as $N(\cU,\cV,\cW)$. 

We begin our thorough examination of $N(\cU, \cV, \cW)$ with an easy upper bound based on multiplicative characters in $\F_q$.

\begin{lemma}\label{lem:N Fq}
Let $\cU, \cV, \cW \subseteq \F_q^*$ with cardinalities $U$, $V$, $W$, respectively. The following inequality holds:
\[
\left| N(\cU, \cV, \cW) -  \frac{U^2 V^2 W^2}{q-1} \right| \le  q UVW.
\]
\end{lemma}
\begin{proof}
Clearly, the number of {\it zero-solutions\/}, that is, solutions  with $u_1(v_1-w_1) = u_2 (v_2-w_2) = 0$ is at most $U^2VW$, because we must have $v_1=w_1$ and $v_2 =w_2$. We use standard properties of multiplicative characters to bound the number $N(\cU, \cV, \cW)^*$ of non-zero solutions.  
Let $\Omega$ denote the set of all $q-1$  multiplicative characters of $\F_q$
and let $\Omega^*$ be the set of nonprincipal characters;
we refer to~\cite[Chapter~3]{IwKow} for a background on characters. In particular, 
using the orthogonality of characters, we write
 \begin{align*}
N(\cU, \cV, \cW)^* &= \frac{1}{q-1}  \sum_{\chi \in \Omega} \sum_{u_1,u_2 \in \cU}  \sum_{v_1,v_2 \in \cV}
 \sum_{w_1,w_2 \in \cW} \chi\left(\frac{u_1 (v_1-w_1)}{u_2 (v_2-w_2)}\right) \\
&  = \frac{1}{q-1} \sum_{\chi \in \Omega} \left|\sum_{u  \in \cU} \sum_{v  \in \cV}  \sum_{w  \in \cW} \chi(u (v-w))\right|^2\\
& = \frac{U^2 V^2 W^2}{q-1} + \frac{1}{q-1} \sum_{\chi \in \Omega^*} \left|\sum_{u  \in \cU } \chi(u )\right|^2
\left| \sum_{\substack{v  \in \cV \\ w  \in \cW}} \chi(v -w )\right|^2.
\end{align*}
Recalling the well-known analogue of~\eqref{eq:bilin}:
\begin{equation}
\label{eq:double char}
 \max _{\chi \in \Omega^*}  
\left| \sum_{\substack{v \in \cV \\ w \in \cW}} \chi(v-w)\right| \le \sqrt{q VW}, 
\end{equation}
we obtain
$$
\left| N(\cU, \cV, \cW)^* -  \frac{U^2 V^2 W^2}{q-1} \right|  \le 
 \frac{q VW }{q-1} \sum_{\chi \in \Omega^*} \left|\sum_{u  \in \cU } \chi(u )\right|^2. 
$$
Using the orthogonality of characters again, we derive
\begin{align*}
\left| N(\cU, \cV, \cW)^* -  \frac{U^2 V^2 W^2}{q-1} \right|  
& \le \ \frac{q VW }{q-1} \left(\sum_{\chi \in \Omega} \left|\sum_{u  \in \cU } \chi(u )\right|^2  - U^2 \right) \\
& \le q UVW - U^2 VW , 
\end{align*}
which concludes the proof. 
\end{proof}

The next step is to obtain a complementary bound for small sets  $\cU, \cV, \cW \subseteq \F_p^*$ in  a prime field $\F_p$. 
It is based on the points-planes incidence bound of  Rudnev~\cite{Rud} 
and in particular on its application described by Aksoy Yazici, Murphy, Rudnev and 
Shkredov~\cite[Theorem~1]{AYMRS}. 
%%PROOFs - numbering of statement in printed version

\begin{lemma}
\label{lem:Collisions}
Let $\cU, \cV, \cW \subseteq \F_p^*$ with cardinalities $U$, $V$, $W$, respectively,  
and set $M =\max\{U,V,W\}$. Suppose that 
$UVW  \ll p^2$. The following inequality holds:
\[
N(\cU, \cV, \cW) \ll U^{3/2}V^{3/2}W^{3/2} + M UVW.
\]
\end{lemma}

\begin{proof}
The result follows from~\cite[Theorem~11]{AYMRS} by setting $A = \cU$ and 
%%PROOFS $L \simeq \cV \times W$ and numbering of statement in printed version
$L \simeq \cV \times \cW$ 
to be the set of lines $\{y = v x + vw~:~v \in \cV, \ w \in \cW\}$.
\end{proof}

Combining the two results gives an upper bound on $N(\cU, \cV, \cW)$ over $\F_p$.

\begin{cor}\label{cor:N Fp}
Let $\cU, \cV, \cW \subseteq \F_p^*$ with cardinalities $U$, $V$, $W$, respectively, and set $M=\max\{U,V,W\}$. The following inequality holds.
\[
N(\cU, \cV, \cW)  \ll \frac{U^2V^2W^2}{p} + U^{3/2}V^{3/2}W^{3/2} + M UVW.
\]
\end{cor}

\begin{proof}
When $UVW \ge p^2$ we apply Lemma~\ref{lem:N Fq}. It is easy to see that
 in this range the $U^2V^2W^2 /p$ term dominates the other two. 
When $UVW <p^2$ we apply Lemma~\ref{lem:Collisions}.
\end{proof}

In particular,  we see from Corollary~\ref{cor:N Fp} that when $U = V = W$ we have 
\[
N(\cU, \cV, \cW)  \ll \frac{U^6}{p} + U^{9/2}.  
\]

We also need two other quantities similar to $N(\cU, \cV,\cW)$. We recall that the {\it multiplicative energy\/} $E_\times (\cU)$ of  a set $\cU \subseteq \F_p$, is defined 
 as the number of solutions to the equation 
 $$
 u_1u_2 =  u_3 u_4, \qquad  u_1,u_2 , u_3, u_4 \in \cU .
 $$
 It is known to play a crucial role in bounds of various exponential sums, 
 see, for example,~\cite{Gar2}.  
 However, here our argument relies on bounds of  the multiplicative energy 
  of the difference set with multiplicities counted.

\begin{definition}\label{def: Dx}
Let $\cU \subseteq \F_q^*$. Then  $D_\times (\cU)$  denotes the number of solutions to the equation  
$$
(u_1-v_1)(u_2 - v_2)=  (u_3-v_3)(u_4-v_4), \quad  u_i,v_i, \in \cU, \ i =1,2,3,4 .
 $$
\end{definition}

An essentially optimal upper bound for $D_\times(\cU)$ over the real numbers has been obtained by Roche-Newton and Rudnev~\cite{RudR-N} by an application of the Guth--Katz theorem~\cite{GuKat2}. The quantity $D_\times(\cU)$ does not seem to have been studied in the finite field context until recently (see~\cite{Pet}). We will present a different bound, which depends on a third quantity that, moreover, features an autonomous role in the proofs of the exponential sum estimates is the following.

\begin{definition}\label{def:T}
Let $\cU \subseteq \F_q^*$. Then  $T (\cU)$  denotes the number of solutions to the equation  
$$
\frac{u_1-v}{u_2 - v}=  \frac{u_3-w}{u_4-w}, \quad  u_i,v, w \in \cU, \ i =1,2,3,4.
 $$
\end{definition}

The notation reflects the fact that the function  $T(\cU)$ counts the number of collinear triplets of points in $\cU \times \cU \subseteq \F_q^2$. A more or less optimal upper bound for $T(\cU)$ over the real numbers has been  obtained by Elekes and Ruzsa in~\cite{EleRuz} by an application of the Szemer\'edi--Trotter theorem~\cite{SzeTro}. In the finite field context, $T(\cU)$ is studied by Aksoy Yazici, Murphy, Rudnev and Shkredov in~\cite{AYMRS}.

We begin our examination of $D_\times(\cU)$ and $T(\cU)$ with the following elementary estimate that relates the two quantities.
 
 \begin{lemma}
\label{lem:DU TU}  
For any set $\cU \subseteq \F_p$ of cardinality $\# \cU = U$, we have
$$
D_\times (\cU)  \ll U^2 T(\cU) + U^6.
$$
\end{lemma}

\begin{proof} Clearly $D_\times (\cU) =  D_\times^* (\cU) +O(U^6)$
where $ D_\times^* (\cU)$  is the number of solutions to the equation 
$$
 \frac{u_1-v_1}{u_2 - v_2}=  \frac{u_3-v_3}{u_4-v_4} \ne 0, \quad  u_i,v_i, \in \cU, \ i =1,2,3,4 .
 $$
Let $J(\lambda)$ be the number of quadruples $(u_1,u_2,v,w) \in \cU^4$ with 
\begin{equation}
\label{eq:Collin}
\frac {u_1-v}{u_2 - w}= \lambda
\end{equation}
and let   $J_{v,w}(\lambda)$  be the number of pairs $(u_1,u_2) \in \cU^4$ 
for which~\eqref{eq:Collin} holds. 

Then, by the Cauchy inequality, we have
\begin{align*}
D_\times^* (\cU) & =  \sum_{\lambda \in \F_p^*} J(\lambda) ^2
=  \sum_{\lambda \in \F_p^*} 
\(\sum_{v,w \in \cU}  J_{v,w}(\lambda)\)^2 \\
& \le U^2 \sum_{\lambda \in \F_p^*} 
\sum_{v,w \in \cU}  J_{v,w}(\lambda)^2  = 
U^2 \sum_{v,w \in \cU}  \sum_{\lambda \in \F_p^*}  J_{v,w}(\lambda)^2. 
\end{align*}
Using that 
\begin{align*}
 \sum_{\lambda \in \F_p^*} J_{v,w}(\lambda)^2 & = \sum_{\lambda \in \F_p^*} \#\left\{(u_1,u_2,u_3,u_4)\in \cU^4~:~\frac{u_1-v}{u_2 - w} = \frac{u_3-v}{u_4 - w} = \lambda\right\}\\
& = \#\left\{(u_1,u_2,u_3,u_4)\in \cU^4~:~\frac{u_1-v}{u_3 - v} = \frac{u_2-w}{u_4 - w}\ne 0\right\}, 
\end{align*}
and renaming the variables $(v,w) \to (u_3,u_2)$, we immediately obtain the desired result. 
\end{proof}

We now need the bound on $T(\cU)$ given by Petridis~\cite{Pet},  which (for $U \le p^{2/3}$) is  based on  a result of Aksoy Yazici, Murphy, Rudnev and 
Shkredov~\cite[Proposition~2.5]{AYMRS} 
%%PROOFs - numbering of statement in printed version
(see also~\cite[Theorem~10]{Shkr}). 

 \begin{lemma}
\label{lem:TU}  
For any set $\cU \subseteq \F_p$ of cardinality $U$, we have
$$
T (\cU)  \ll \frac{U^6}{p} + U^{9/2}.
$$
\end{lemma}

Thus combining Lemmas~\ref{lem:DU TU} and~\ref{lem:TU} we obtain 

 \begin{cor}
\label{cor:DU}  
For any set $\cU \subseteq \F_p$ of cardinality $U$, we have
$$
D_\times (\cU)  \ll \frac{U^8}{p} + U^{13/2}. 
$$
\end{cor}

In particular, we see from Corollary~\ref{cor:DU} that when $U \le p^{2/3}$ then 
the second term dominates and we derive
$D_\times(\cU) \ll U^{13/2}$. We also remark that 
Rudnev,   Shkredov and Stevens~\cite[Theorem~6]{RSS} have given a related  result.

 \subsection{Exponential sums and differences}
 \label{sec:ExpEnergy}

We now link  multilinear exponential sums~\eqref{eq:Sum T} to exponential sums with differences. For  a technical reason 
it is also convenient  for us to prove this in the setting of arbitrary finite fields and additive characters. This may also be 
useful in further applications. 

Namely, 
 for sets $\cX_1,\dots, \cX_n \subseteq \F_q$, weights $\omega_1,\dots,\omega_n$ and an additive character $\psi$ 
 of $\F_q$, 
we define 
$$
T_\psi(\cX_1, \ldots, \cX_n;  \omega_1, \ldots, \omega_n) = \ssum_{\vec{x} \in\cX_1\times \ldots \times \cX_n} 
\omega_1(\vec{x}) \ldots \omega_n(\vec{x}) \psi(x_1 \ldots x_n).
$$
Note that  our bound is  uniform in $\psi$    (and  is actually trivial when $\psi=\psi_0$ is the principal character).

 \begin{lemma}
\label{lem:ExpSum} 
Let $n \ge 2$. For any additive character $\psi$ 
 of $\F_q$, sets $\cX_i \subseteq \F_q$ of cardinalities $\# \cX_i = X_i$, and weights 
 $\omega_i= (\omega_i(\vec{x}))_{\vec{x} \in \F_p^n}$ such that $\omega_i(\vec{x})$ does not 
depend on the $i$th coordinate of   $\vec{x} = (x_1, \ldots, x_n)$ 
and
$$
\max_{\vec{x} \in \F_p^n} |\omega_i(\vec{x})| \le 1, 
$$
for $i =1, \ldots, n$, we have
\begin{equation*}
\begin{split}
|T_\psi(\cX_1&, \ldots, \cX_n;  \omega_1, \ldots, \omega_n)|^{2^{n-1}}\\
&\le  X_1^{2^{n-1}-1}(X_2\ldots X_{n})^{2^{n-1}-2}  \sum_{x_2, y_2 \in\cX_2}  \ldots \sum_{x_{n}, y_{n} \in\cX_{n}} \\
& \qquad \qquad \qquad  \qquad   \quad  \left| \sum_{x_1 \in\cX_1}   \psi( x_1 (x_2-y_2) \ldots (x_{n}-y_{n})) \right|.
\end{split}
\end{equation*}
\end{lemma}

\begin{proof}
For every complex number $\xi$ we write $|\xi|^2 = \xi \overline \xi$.

We first establish the result for $n=2$. Let us begin by eliminating $\omega_2$ by applying the triangle inequality.
\begin{align*}
|T_\psi(\cX_1, \cX_2; \omega_1, \omega_2)|
& =  \left| \sum_{x_1 \in \cX_1}  \omega_2(x_1) \sum_{x_2 \in \cX_2} \omega_1(x_2)   \psi(x_1x_2) \right| \\
&  \le \sum_{x_1 \in \cX_1} \left|  \omega_2(x_1)\right|  \left| \sum_{x_2 \in \cX_2} \omega_1(x_2)   \psi(x_1x_2) \right|.
\end{align*}
Squaring both sides and applying the Cauchy inequality gives
\begin{align*}
|T_\psi(\cX_1, \cX_2; \omega_1, \omega_2)|^2 
&  \le  X_1 \sum_{x_1 \in \cX_1}  \left| \sum_{x_2 \in \cX_2} \omega_1(x_2)   \psi(x_1x_2) \right|^2 \\
& =  X_1  \sum_{x_1 \in \cX_1}  \sum_{x_2, y_2 \in \cX_2} \omega_1(x_2) \overline{\omega_1(y_2)}  \psi(x_1(x_2-y_2)) \\
& =  X_1  \sum_{x_2, y_2 \in \cX_2} \omega_1(x_2) \overline{\omega_1(y_2)} \sum_{x_1 \in \cX_1}   \psi(x_1(x_2-y_2)) \\
& \le  X_1  \sum_{x_2, y_2 \in \cX_2} \left| \sum_{x_1 \in \cX_1}   \psi(x_1(x_2-y_2))\right|.
\end{align*}
This proves the result for $n=2$. For $n>2$ the argument is similar, though it requires the H\"older inequality.  Using that $ \omega_n(\vec{x})$ does not depend on $x_n$, we get 
\begin{equation*}
\begin{split}
|T_\psi(\cX_1&, \ldots, \cX_n;  \omega_1, \ldots, \omega_n)|\\
& \le  \sum_{x_1\in\cX_1} \ldots \sum_{x_{n-1}\in  \cX_{n-1}}  \left|\sum_{x_n \in \cX_n}
\omega_1(\vec{x}) \ldots \omega_{n-1}(\vec{x})  \psi(x_1 \ldots x_n)\right|.
\end{split}
\end{equation*}
Hence, by the Cauchy inequality we derive 
\begin{equation*}
\begin{split}
|T_\psi(\cX_1&, \ldots, \cX_n;  \omega_1, \ldots, \omega_n)|^2\\
& \le  X_1\ldots X_{n-1} \sum_{x_1\in\cX_1} \ldots \sum_{x_{n-1}\in  \cX_{n-1}}  \\
& \qquad  \qquad \qquad    \left|\sum_{x_n \in \cX_n}
\omega_1(\vec{x}) \ldots \omega_{n-1}(\vec{x})  \psi( x_1 \ldots x_n)\right|^2\\
&=  X_1\ldots X_{n-1} \sum_{x_1\in\cX_1} \ldots \sum_{x_{n-1}\in  \cX_{n-1}}  \\
& \qquad  \qquad \qquad \sum_{x_n, y_n \in \cX_n}
\omega_1(\vec{x})  \overline  \omega_1(\vec{y})  \ldots \omega_{n-1}(\vec{x})  \overline  \omega_{n-1}(\vec{y})\\
& \qquad  \qquad \qquad  \qquad  \qquad \qquad \qquad \psi\( x_1 \ldots x_{n-1} (x_n-y_n)\),
\end{split}
\end{equation*}
where $\vec{y} =  (x_1, \ldots, x_{n-1}, y_n)$ (and as before $\vec{x} = (x_1, \ldots, x_n)$). 
By changing the order of summation, we obtain
$$
|T_\psi(\cX_1, \ldots, \cX_n;  \omega_1, \ldots, \omega_n)|^2 
 \le   X_1\ldots X_{n-1} \sum_{x_n, y_n \in \cX_n}   \fT(x_n,y_n), 
 $$
where 
\begin{align*} 
 \fT(x_n,y_n) &=  \sum_{x_1\in\cX_1} \ldots \sum_{x_{n-1}\in  \cX_{n-1}}  
  \omega_1(\vec{x})  \overline  \omega_1(\vec{y})  \ldots \omega_{n-1}(\vec{x})  \overline  \omega_{n-1}(\vec{y})\\
& \qquad  \qquad \qquad  \qquad  \qquad \qquad  \quad  \psi\(x_1 \ldots x_{n-1} (x_n-y_n)\). 
\end{align*}

Now, raising both sides to the $2^{n-2}$th power and using the H{\"o}lder inequality, we obtain
\begin{equation}
\label{eq:Induct}
\begin{split}
|T_\psi(\cX_1&, \ldots, \cX_n;  \omega_1, \ldots, \omega_n)|^{2^{n-1}}\\
& \le   \(X_1\ldots X_{n-1}\)^{2^{n-2}} X_n^{2(2^{n-2}-1)} \sum_{x_n, y_n \in \cX_n} \left|\fT(x_n,y_n)\right|^{2^{n-2}}.
\end{split}
\end{equation}

We apply the inductive hypothesis to bound $\left|\fT(x_n,y_n)\right|^{2^{n-2}}$
for every $x_n, y_n \in \cX_n$. For each $ j =1, \ldots, n-1$, the weights
$$
\omega_j(\vec{x})  \overline \omega_j(\vec{y}) =
\omega_j((x_1, \ldots, x_{n-1}, x_n)) \overline \omega_j((x_1, \ldots, x_{n-1}, y_n)) 
$$
satisfy the necessary conditions with respect to $x_1, \ldots, x_{n-1}$. Applying the inductive hypothesis to the above
 character sum (where
the character $\psi(u)$ is replaced by the character $\psi_{x_n, y_n}(u) = \psi(u (x_n-y_n))$) gives

\begin{equation*}
\begin{split}
\left|\fT(x_n,y_n)\right|^{2^{n-2}}
& 
\le  X_1^{2^{n-2}-1}(X_2\ldots X_{n-1})^{2^{n-2}-2}  \sum_{x_2, y_2 \in\cX_2}  \ldots \sum_{x_{n-1}, y_{n-1} \in\cX_{n-1}} \\
 & \qquad  \left| \sum_{x_1 \in\cX_1}    \psi( x_1 (x_2-y_2) \ldots (x_{n-1}-y_{n-1}) (x_n-y_n)) \right|.
\end{split}
\end{equation*}
Substituting in~\eqref{eq:Induct} completes the inductive step and thus finishes the proof. 
 \end{proof}

\section{Proofs of bounds of exponential sums}
 
 \subsection{Proof of Theorem~\ref{thm:Bound S3}}

Recall that
\[
S(\cX, \cY, \cZ;  \alpha, \beta, \gamma) = \sum_{x \in \cX} \sum_{y \in \cY} \sum_{z \in \cZ} \alpha_x \beta_y \gamma_z \ep(xyz).
\]
We begin by eliminating $\alpha_x$ and $\gamma_z$.
\begin{align*}
 |S(\cX, \cY, \cZ;  \alpha, \beta, \gamma)|  
 &  = \left|\sum_{x \in \cX} a_x \sum_{z \in \cZ} \gamma_z \sum_{y \in \cY}  \beta_y  \ep(xyz) \right|\\
 & \le   \sum_{x \in\cX} \sum_{z\in \cZ}  \left|  \sum_{y \in \cY} \beta_{y}  \ep(xyz)\right|.
\end{align*}
We proceed as in the proof of Lemma~\ref{lem:ExpSum}  and, using the Cauchy inequality, derive 
\begin{align*}
|S(\cX, \cY, \cZ&;  \alpha, \beta, \gamma)|^2 \le XZ \sum_{x \in\cX} \sum_{z \in \cZ}
\left| \sum_{y\in \cY}   \beta_{y}  \ep(xyz)\right|^2
\\
 &= XZ \sum_{x \in\cX}  \sum_{z \in \cZ} \sum_{y_1, y_2 \in \cY} \beta_{y_1}   \overline\beta_{y_2} \ep(xz(y_1-y_2)).\end{align*}
For $\lambda \in \F_p$, we now denote 
$$
\widetilde J(\lambda) =
\sum_{\substack{(y_1, y_2, z) \in \cY^2 \times \cZ\\ (y_1-y_2) z = \lambda}} \beta_{y_1}   \overline\beta_{y_2}
$$
thus we have 
$$
|S(\cX, \cY, \cZ;  \alpha, \beta, \gamma)|^2 \le XZ \sum_{\lambda \in \F_p} \sum_{x \in\cX} \widetilde J(\lambda)   \ep(\lambda x).
$$
Clearly, $\left| \widetilde J(\lambda)\right| \le J(\lambda)$, where $J(\lambda)$ is 
the number of triples  $(y_1,y_2,z ) \in \cY^2 \times \cZ$ with the same 
value of the product $(y_1-y_2)z = \lambda \in \F_p$. It is also clear that 
\begin{equation}
\label{eq:J and N}
\sum_{\lambda \in \F_p} J(\lambda) ^2 = N(\cZ, \cY, \cY),
\end{equation}
where $N(\cU, \cV,\cW)$ is as in Definition~\ref{def:N} 
in Section~\ref{sec:AddComb}.

Thus applying~\eqref{eq:bilin} together with~\eqref{eq:J and N} 
we obtain
$$
|S(\cX, \cY, \cZ;  \alpha, \beta, \gamma)|^2 \le XZ \sqrt{p X N(\cZ, \cY, \cY) }. 
$$
Using Corollary~\ref{cor:N Fp}, we derive
$$
 S(\cX, \cY, \cZ;  \alpha, \beta, \gamma)  \ll   X^{3/4} Y Z + p^{1/4} X^{3/4} Y^{3/4} Z^{7/8}. 
$$
Our final task is to remove the first term  $X^{3/4} YZ$. 
This first term dominates the second term only when $ p \le YZ^{1/2}$. In this range, however, we deploy the classical bound following from the triangle inequality and~\eqref{eq:bilin}
\begin{align*}
|S(\cX, \cY, \cZ;  \alpha, \beta, \gamma)  |
& \le  p^{1/2} X^{1/2} Y^{1/2} Z  \le p^{1/4} (YZ^{1/2})^{1/4} X^{1/2} Y^{1/2} Z \\ 
& = p^{1/4} X^{1/2} Y^{3/4} Z^{9/8}  =
p^{1/4} X^{3/4} Y^{3/4} Z^{7/8} (Z/X)^{1/4} \\
&  \le p^{1/4} X^{3/4} Y^{3/4} Z^{7/8}.
\end{align*}
Therefore the term $X^{3/4} Y Z$ can be dropped from the final bound.

  \subsection{Proof of Theorem~\ref{thm:Bound S4}}

As in the proof of Theorem~\ref{thm:Bound S3}, after an application of  the Cauchy inequality,
we arrive to
\begin{align*}
|S(\cW, \cX, \cY, \cZ&;\alpha, \beta, \gamma, \delta)  |^2  \\
 & \le WYZ \sum_{w \in\cW} \sum_{y \in \cY}  \sum_{z \in \cZ} \sum_{x_1, x_2 \in \cX} \beta_{x_1}   \overline\beta_{x_2} \ep(wyz(x_1-x_2))\\
  & \le WYZ \sum_{x_1, x_2 \in \cX} \sum_{y \in\cY}  \sum_{z \in \cZ} 
  \left|  \sum_{w \in\cW}  \ep(wyz(x_1-x_2))\right|.
\end{align*}

Applying  the Cauchy inequality one more time, we derive
\begin{equation}
\label{eq:S4-4}
\begin{split}
|S(\cW&, \cX, \cY, \cZ;\alpha, \beta, \gamma, \delta)|^4  \\
&  \le (WYZ)^2 X^2 Y Z \sum_{x_1, x_2 \in \cX} \sum_{y \in\cY}  \sum_{z \in \cZ} 
  \left|  \sum_{w \in\cW}  \ep(wyz(x_1-x_2))\right|^2\\
&  =  W^2 X^2 Y^3Z^3  
 \sum_{w_1, w_2 \in \cX} \sum_{x_1, x_2 \in \cX} \sum_{y \in\cY} \\
 & \qquad\qquad\qquad\qquad  \sum_{z \in \cZ} 
  \ep(yz(w_1-w_2)(x_1-x_2)). 
  \end{split}
\end{equation}

We now collect together triples  $(w_1,w_2, y) \in\cW^2\times \cZ$ with the same 
value of the product $y(w_1-w_2) = \lambda \in \F_p$ and denote by $I(\lambda)$ 
the number of such  triples.   

Similarly, we  collect together   triples  $(x_1,x_2,z) \in \cX^2 \times \cZ$ with the same 
value of the product  $z(x_1-x_2) = \mu \in \F_p^*$ and denote by $J(\mu)$ 
the number of such  triples.

Hence we can rewrite~\eqref{eq:S4-4} as
\begin{equation}
\label{eq:S4 IJ}
\begin{split}
S(\cW, \cX, \cY, \cZ&;\alpha, \beta, \gamma, \delta)^4  \\
&  \ll  W^2 X^2 Y^3Z^3    \sum_{\lambda, \mu  \in \F_p} I(\lambda)  
J(\mu)   \ep( \lambda \mu).
  \end{split}
\end{equation}
Using analogues of the identity~\eqref{eq:J and N}, we note that  by Lemma~\ref{lem:Collisions} and by the assumption $p^{2/3} \ge W \ge X \ge Y \ge Z$,
$$
\sum_{\lambda \in \F_p} I(\lambda) ^2  \ll   W^3 Y^{3/2} +  W^3 Y   \ll   W^3 Y^{3/2} 
$$
and 
$$ \sum_{\mu \in \F_p} J(\mu) ^2  \ll  X^3 Z^{3/2} +  X^3 Z   \ll   X^3 Z^{3/2} .
$$
Therefore, using the bound on bilinear sums~\eqref{eq:bilin}, 
 we obtain 
\begin{align*}
S(\cW, \cX, \cY, \cZ;\alpha, \beta, \gamma, \delta)^{4}  & \ll 
 W^2 X^2 Y^3Z^3 \sqrt{ p W^3 X^3 Y^{3/2}  Z^{3/2}},\\
 & = p^{1/2} W^{7/2} X^{7/2} Y^{15/4}  Z^{15/4}, 
\end{align*}
which concludes the proof.

  \subsection{Proof of Theorem~\ref{thm:Bound T3}}
 
 First we observe that by the classical bound~\eqref{eq:bilin} we have
 $$
T(\cX, \cY, \cZ;  \rho, \sigma, \tau) \ll \sqrt{pXY} Z,
 $$
which implies Theorem~\ref{thm:Bound T3} provided that 
$$
  p^{1/8} X^{7/8} Y^{29/32}  Z^{29/32}  \ge  \sqrt{pXY} Z
$$
or, equivalently, 
$$
 X^{3/8} Y^{13/32}  Z^{-3/32}  \ge    p^{3/8}.
$$
We now note that for $X \ge Y> p^{2/3}$ we have 
$$
 X^{3/8} Y^{13/32}  Z^{-3/32}  \ge    X^{3/8} Y^{5/16}   \ge Y^{11/16} \ge p^{11/24} > p^{3/8} .
$$
For this reason, we can assume that
\begin{equation}
\label{eq:Y small}
Z \le Y \le p^{2/3}.
\end{equation}

Now, by Lemma~\ref{lem:ExpSum} (with $n=3$ and $\psi = \ep$) we have
\begin{align*}
|T(\cX&, \cY, \cZ;  \rho, \sigma, \tau)|^4  \\
  &   \le X^3Y^2Z^2  \sum_{y_1,y_2 \in \cY} 
  \sum_{z_1,z_2\in \cZ}   \left| \sum_{x \in\cX}   \ep(x(y_1-y_2)(z_1-z_2))\right|. 
\end{align*}

The number of quadruples $(y_1,y_2,z_1,z_2) \in \cY^2 \times \cZ^2$ which satisfy $(y_1-y_2)(z_1-z_2) = 0$ is at most $YZ^2 + Y^2Z \le 2 Y^2Z$. For such quadruples the inner sums is equal to $X$. Hence, 
\begin{equation}
\label{eq:NoWeights}
\begin{split}
T(\cX&, \cY, \cZ;  \rho, \sigma, \tau)^4  \\
  &   \ll X^3Y^2Z^2  \sum_{\substack{y_1,y_2 \in \cY\\ y_1\ne y_2}} 
  \sum_{\substack{z_1,z_2\in \cZ\\z_1\ne z_2}}   \left| \sum_{x \in\cX}   \ep(x(y_1-y_2)(z_1-z_2))\right|\\
 & \qquad  \qquad  \qquad  \qquad  \qquad  \qquad  \qquad  \qquad  \qquad  \qquad  + X^4Y^4Z^3.
\end{split}
\end{equation}

We now collect together quadruples   $(y_1,y_2,z_1,z_2) \in \cY^2 \times \cZ^2$ with the same 
value of the product $(y_1-y_2)(z_1-z_2) = \lambda \in \F_p^*$ and denote by $J(\lambda)$ 
the number of such quadruples. Hence we can rewrite~\eqref{eq:NoWeights} as
\begin{equation}
\label{eq:T J}
T(\cX, \cY, \cZ;  \rho, \sigma, \tau)^4     \ll X^3Y^2Z^2  \sum_{\lambda \in \F_p^*} J(\lambda)  \left| \sum_{x \in\cX}   \ep( \lambda x)\right| + X^4Y^4Z^3. 
\end{equation}
Yet another application of the Cauchy inequality leads to the bound 
$$
T(\cX, \cY, \cZ;  \rho, \sigma, \tau)^8    \ll X^6Y^4Z^4 K
\sum_{\lambda \in \F_p}  \left| \sum_{x \in\cX}   \ep( \lambda x)\right|^2
+ X^8Y^8Z^6, 
$$
where 
$$
K =   \sum_{\lambda \in \F_p^*} J(\lambda) ^2.
$$
By the orthogonality of exponential functions, we get 

\begin{equation}
\label{eq:orthog}
\sum_{\lambda \in \F_p}  \left| \sum_{x \in\cX}   \ep( \lambda x)\right|^2 
= p X.
\end{equation}
Thus we have
\begin{equation}
\label{eq:T3-8}
T(\cX, \cY, \cZ;  \rho, \sigma, \tau)^8    \ll p X^7Y^4Z^4 K  + X^8Y^8Z^6.
\end{equation}

It now remains to estimate $K$.  Clearly
$K$ is the number of solutions to the equation 
\begin{align*}
(y_1-y_2)&(z_1-z_2) = (y_3-y_4)(z_3-z_4) \ne 0, \\ \
 (y_i&,z_i) \in \cY\times \cZ, \quad i =1,2,3,4.
\end{align*}
We express $K$ in terms of multiplicative characters as follows. See~\cite[Chapter~3]{IwKow} for a background on multiplicative
characters.
$$
K= \dsum_{\substack{y_1,y_2, y_3, y_4 \in \cY\\
z_1,z_2, z_3, z_4 \in \cZ}} \frac{1}{p-1}\sum_{\chi \in \Omega} \chi\((y_1-y_2)(z_1-z_2)\) 
\overline \chi\((y_3-y_4)(z_3-z_4)\).
$$
The inner sums is over all $p-1$ distinct multiplicative characters  $\chi$ of $\F_p$. 
Simple transformations lead to the formula
$$
K = \frac{1}{p-1}\sum_{\chi \in \Omega} \left |\sum_{ y_1,y_2,\in \cY}  \chi(y_1-y_2)\right|^2
\left| \sum_{z_1,z_2 \in \cZ} \chi(z_1-z_2)\right|^2. 
$$
Now, using the Cauchy inequality, we obtain 
$$
K^2 \le \sum_{\chi \in \Omega} \left |\sum_{ y_1,y_2,\in \cY}  \chi(y_1-y_2)\right|^4 \, \sum_{\chi \in \Omega} \left| \sum_{z_1,z_2 \in \cZ} \chi(z_1-z_2)\right|^4 \le D_\times (\cY) D_\times (\cZ),
$$
where $D_\times (\cU)$ is as in Definition~\ref{def: Dx} in Section~\ref{sec:AddComb},
see~\cite[Lemma~4]{BGKS} for a similar argument. 
Recalling Corollary~\ref{cor:DU},  we obtain 
\begin{equation}
\label{eq:K crude}
K \ll \(Y^4p^{-1/2} +  Y^{13/4}\) \(Z^4p^{-1/2} +  Z^{13/4}\).
\end{equation}
One easily verifies that under the condition~\eqref{eq:Y small} we have 
$$
Y^4p^{-1/2} \le Y^{13/4}
\mand 
Z^4p^{-1/2} \le  Z^{13/4}.
$$ 
Hence, the bound~\eqref{eq:K crude} simplifies as 
\begin{equation}
\label{eq:K fin}
K \ll    Y^{13/4} Z^{13/4}.
\end{equation}
Using this bound together with~\eqref{eq:T3-8}, we obtain
$$
T(\cX, \cY, \cZ;  \rho, \sigma, \tau)^8     \ll   p X^7 Y^{29/4}  Z^{29/4} 
   +    X^8Y^8Z^6. 
$$

This implies
\[
T(\cX, \cY, \cZ;  \rho, \sigma, \tau)     \ll   p^{1/8} X^{7/8} Y^{29/32}  Z^{29/32} 
   +    XYZ^{3/4}. 
\] 
Our final task is to remove the second summand from the upper bound. The second summand dominates the first when $p^4 Z^5 \le X^4 Y^3$ or equivalently $(p/X)^4 \le Y^3 / Z^5$. In this range we apply the triangle inequality and the bilinear estimate~\eqref{eq:bilin}
\begin{align*}
|T(\cX, \cY, \cZ;  \rho, \sigma, \tau)|   
& \le p^{1/2} X^{1/2} Y^{1/2} Z \\
& = p^{1/8} X^{7/8} Y^{29/32} Z^{29/32} \left(\frac{p^{12}Z^{3}}{X^{12}Y^{13}}\right)^{1/32} \\
& \le p^{1/8} X^{7/8} Y^{29/32} Z^{29/32} \left(\frac{1}{Y^{4} Z^{12}}\right)^{1/32} \\
& \le p^{1/8} X^{7/8} Y^{29/32} Z^{29/32}. 
\end{align*}

  \subsection{Proof of Theorem~\ref{thm:Bound T4}}
 
First we permute the variables and write  
$$
T(\cW, \cX, \cY, \cZ; \vartheta,  \rho, \sigma, \tau)
= 
T(\cZ, \cW, \cX, \cY;  \tau, \vartheta,  \rho,  \sigma)
$$
Then, by Lemma~\ref{lem:ExpSum} 
(with $n=4$)  we have
\begin{equation}
\begin{split}
\label{eq:T4-8}
|T(\cW, \cX&, \cY, \cZ; \vartheta,  \rho, \sigma, \tau) |^8  \\   &\le (WXY)^6   Z^7 \sum_{w_1,w_2\in\cW}  \sum_{x_1,x_2 \in \cX} \\
&  \qquad\quad \sum_{y_1,y_2 \in \cY}   \left| \sum_{z \in\cZ}   \ep(z(w_1-w_2)(x_1-x_2)(y_1-y_2))\right|. 
\end{split}
\end{equation}

We separate the cases where in the argument in the exponential  we have $x_1 = x_2$ or $y_1 = y_2$. 
Recalling that $W\ge X \ge Y\ge Z$, we see there are $O(W^2XY^2Z + W^2X^2YZ) = O(W^2 X^2 YZ)$ such terms, each contributing 1 to the sum. We can thus rewrite~\eqref{eq:T4-8} as 
\begin{align*}
T(\cW, \cX, \cY, \cZ ; \vartheta&,  \rho, \sigma, \tau)^{8}  \\   \ll  (WXY)^6  & Z^7   
\sum_{w_1,w_2\in\cW}  \sum_{\substack{x_1,x_2 \in \cX\\x_1 \ne x_2}} \\
&  \qquad   \sum_{\substack{y_1,y_2 \in \cY\\y_1 \ne y_2}}
   \left| \sum_{z \in\cZ}   \ep(z(w_1-w_2)(x_1-x_2)(y_1-y_2))\right|   \\
& \qquad \qquad \qquad \qquad \qquad  \qquad \qquad\qquad+ (WXZ)^8 Y^7. 
\end{align*}

We now collect together triples  $(w_1,w_2, z) \in\cW^2\times \cZ$  with the same 
value of the product $z(w_1-w_2) = \lambda \in \F_p$ and denote by $I(\lambda)$ 
the number of such triples.   

Similarly, we  collect together quadruples   $(x_1,x_2,y_1,y_2) \in \cX^2 \times \cY^2$ with the same 
value of the product $(x_1-x_2)(y_1-y_2) = \mu \in \F_p^*$ and denote by $J(\mu)$ 
the number of such quadruples. 

Hence we obtain the following version of~\eqref{eq:S4 IJ}
\begin{align*}
T(\cW, \cX&, \cY, \cZ; \vartheta, \rho, \sigma, \tau)^{8}  \\
&   \ll   (WXY)^6   Z^7 \sum_{\mu \in \F_p^*} 
J(\mu)  \left|  \sum_{\lambda \in \F_p} I(\lambda)  \ep( \lambda \mu)\right|  + (WXZ)^{8}Y^{7}  
\\
&   \ll    (WXY)^6   Z^7 \sum_{\mu \in \F_p^*}   \sum_{\lambda \in \F_p} 
J(\mu)  \eta_\mu I(\lambda)  \ep( \lambda \mu)  +  (WXZ)^{8}Y^{7} , 
\end{align*}
%%PROOFS Comma after  $\mu \in \F_p^*$
where $\eta_\mu$,  $\mu \in \F_p^*$, is a complex number with $|\eta_\mu|=1$ 
(which can be expressed via the argument of the inner sum). 
We note that  by Lemma~\ref{lem:Collisions} and by the assumption $p^{2/3} \ge W \ge X \ge Y \ge Z$,
$$
\sum_{\lambda \in \F_p} I(\lambda) ^2  \ll   W^3 Z^{3/2} +  W^3 Z   \ll   W^3 Z^{3/2} .
$$
Similarly to~\eqref{eq:K fin} we also have 
$$ \sum_{\mu \in \F_p^*} J(\mu) ^2  \ll    X^{13/4} Y^{13/4}.
$$
Therefore, using the bound on bilinear sums~\eqref{eq:bilin}, 
 we obtain 
$$
T(\cW, \cX, \cY, \cZ; \vartheta, \rho, \sigma, \tau)^{8}     \ll p^{1/2}  W^{15/2} (XY)^{61/8}   Z^{31/4} 
+ (WXZ)^{8}Y^{7}  ,
$$
which implies
\[
T(\cW, \cX, \cY, \cZ; \vartheta, \rho, \sigma, \tau)  \ll p^{1/16}  W^{15/16} (XY)^{61/64}   Z^{31/32} 
+ W X Y^{7/8} Z.
\]
Our final task is to remove the second summand. The process is identical to that at the end of Theorem~\ref{thm:Bound S3} and of Theorem~\ref{thm:Bound T3}. The second summand dominates when $p^4 Y^5 \le W^4 X^3 Z^2$. However, in this range we apply the triangle inequality and the bilinear bound~\eqref{eq:bilin} to obtain
\[
|T(\cW, \cX, \cY, \cZ; \vartheta, \rho, \sigma, \tau)| \le p^{1/2} W^{1/2} X^{1/2} YZ.
\]
The condition $p^4 Y^5 \le W^4 X^3 Z^2$ implies
\[
p^{1/2} W^{1/2} X^{1/2} YZ \le p^{1/16}  W^{15/16} (XY)^{61/64}   Z^{31/32} 
\]
and the proof is completed.

\section{Proofs of expansion of polynomial images}

\subsection{General approach}

Our proof strategy is modeled from that of Garaev~\cite{Gar1}. Namely, each time 
we consider the number of solutions to an appropriately chosen equation 
and use our new bounds of exponential sums to obtain an asymptotic formula for the number
of its solutions. On the other hand, 
using the specially designed structure of this equation,
we are able to explicitly produce a large family of solutions. Both the  asymptotic formula
and the size of the family of explicit solutions depend on the cardinalities of the image sets
of interest and lead to the desired results.

  \subsection{Proof of Theorem~\ref{thm:Image1}}
  
We let  $\cE = \F_p\setminus \( \cA\cB\cC  + \cD\)$ be of cardinality $E$ and  
define $N$ as the number of solutions to the equation
$$
abc+d -e = 0, \qquad (a,b,c,d,e) \in \cA\times \cB \times\cC\times \cD \times\cE, 
$$
for which we obviously  have $N = 0$. On the other hand, using the orthogonality of  exponential functions 
we express  $N$ as
$$
N = \ssum_{(a,b,c,d,e) \in \cA\times \cB \times\cC\times \cD \times\cE} \frac{1}{p} \sum_{\lambda \in \F_p}
\ep\(\lambda(abc +d -e)\).
$$
Changing the order of summation and separating the term $ABCDE/p$ corresponding to $\lambda = 0$, 
we obtain 
$$
N =  \frac{ABCDE}{p}  + \frac{1}{p} \sum_{\lambda \in \F_p^*}  \sum_{(a,b,c) \in \cA\times \cB \times\cC}
\ep(\lambda abc ) \sum_{d \in   \cD }\ep(\lambda d)  \sum_{e \in   \cE} \ep(-\lambda e). 
$$
Hence, applying Theorem~\ref{thm:Bound S3}, we derive 
$$
N - \frac{ABCDE}{p}    \ll p^{-3/4} A^{3/4} B^{3/4}  C^{7/8}   \sum_{\lambda \in \F_p^*} \left|
 \sum_{d \in   \cD }\ep(\lambda d)\right|  \left| \sum_{e \in   \cE} \ep(\lambda e)\right|.
 $$
Now applying the Cauchy inequality and then  analogues  of the orthogonality relation~\eqref{eq:orthog}, we obtain 
\begin{equation}
\begin{split}
\label{eq:N ABCDE}
&N - \frac{ABCDE}{p} \\
 &\quad  \ll p^{-3/4} A^{3/4} B^{3/4}  C^{7/8}   \(\sum_{\lambda \in \F_p} \left|
 \sum_{d \in   \cD }\ep(\lambda d)\right|^2  \sum_{\lambda \in \F_p}  \left| \sum_{e \in   \cE} \ep(\lambda e)\right|^2\)^{1/2}\\
 &\quad = p^{-3/4} A^{3/4} B^{3/4}  C^{7/8}    (pD)^{1/2} (pE)^{1/2}\\
 &\quad  =  p^{1/4} A^{3/4} B^{3/4}  C^{7/8}  D^{1/2} E^{1/2}.
\end{split}
\end{equation}
Recalling that $N=0$  we obtain
$$
ABCDE  \ll  p^{5/4} A^{3/4} B^{3/4}  C^{7/8}  D^{1/2} E^{1/2},
 $$
which concludes the proof of the asymptotic formula.

For the lower bound,  we consider the number $J(\eta)$ of solutions to the equation
$$
abc+d = \eta, \qquad (a,b,c,d) \in \cA\times \cB \times\cC\times \cD. 
$$
Clearly 
\begin{equation}
\label{eq:Sum J}
 \sum_{\eta \in \F_p} J(\eta) = ABCD.
\end{equation}
Furthermore, by the Cauchy inequality
\begin{equation}
\label{eq:Cauchy}
\(\sum_{\eta \in \F_p} J(\eta)\)^2 \le  \#\(\cA\cB \cC +\cD\) \sum_{\eta \in \F_p} J(\eta)^2
=  \#\(\cA\cB \cC +\cD\)  J, 
\end{equation}
where $J$ is the  number  of solutions to the equation
$$
a_1b_1c_1+d_1  = a_2b_2c_2+d_2, \quad (a_\nu,b_\nu,c_\nu, d_\nu) \in \cA\times \cB \times\cC\times \cD, \ \nu = 1,2.
$$
As before, we write 
\begin{align*}
J & =  \ssum_{\substack{(a_\nu,b_\nu,c_\nu, d_\nu) \in \cA\times \cB \times\cC\times \cD\\\nu =1,2}} 
\frac{1}{p} \sum_{\lambda \in \F_p}
\ep\(\lambda(a_1b_1c_1 - a_2b_2c_2+d_1 -d_2)\)\\
& = \frac{1}{p} \sum_{\lambda \in \F_p}
\left| \ssum_{(a,b,c) \in \cA\times \cB \times\cC} \ep\(\lambda abc\)\right|^2 
 \left| \sum_{d \in   \cD }\ep(\lambda d)\right|^2\\
 & =  \frac{A^2B^2C^2D^2}{p} + \frac{1}{p} \sum_{\lambda \in \F_p^*}
\left| \ssum_{(a,b,c) \in \cA\times \cB \times\cC} \ep\(\lambda abc\)\right|^2 
 \left| \sum_{d \in   \cD }\ep(\lambda d)\right|^2.
\end{align*}
We now recall Theorem~\ref{thm:Bound S3} and derive 
\begin{align*}
J &  \ll   \frac{A^2B^2C^2D^2}{p} + p^{-1/2} A^{3/2} B^{3/2}  C^{7/4} \sum_{\lambda \in \F_p}
 \left| \sum_{d \in   \cD }\ep(\lambda d)\right|^2\\
 & = \frac{A^2B^2C^2D^2}{p} + p^{1/2} A^{3/2} B^{3/2}  C^{7/4} D.
\end{align*}
Substituting this bound in~\eqref{eq:Cauchy} and recalling~\eqref{eq:Sum J}
we obtain the desired lower bound. 

  \subsection{Proof of Theorem~\ref{thm:Image2}}
  
We proceed as   in the proof of Theorem~\ref{thm:Image1}.

  Let $\cE = \F_p\setminus \((\cA+\cB+\cC)^3 + \cD\)$ be of cardinality $E$. 
Then for the number $N$ of solutions to the equation
$$
(a+b+c)^3 +d -e = 0, \qquad (a,b,c,d,e) \in \cA\times \cB \times\cC\times \cD \times\cE
$$
we obviously have $N = 0$.
On the other hand, using the orthogonality of  exponential functions we express  $N$ as
$$
N = \ssum_{(a,b,c,d,e) \in \cA\times \cB \times\cC\times \cD \times\cE} \frac{1}{p} \sum_{\lambda \in \F_p}
\ep(\lambda((a+b+c)^3 +d -e).
$$
Changing the order of summation and separating the term $ABCDE/p$ corresponding to $\lambda = 0$, 
we obtain 
\begin{align*}
N = & \frac{ABCDE}{p}\\
&\quad  + \frac{1}{p} \sum_{\lambda \in \F_p^*}  \sum_{(a,b,c) \in \cA\times \cB \times\cC}
\ep\(\lambda(a+b+c)^3\) \sum_{d \in   \cD }\ep(\lambda d)  \sum_{e \in   \cE} \ep(-\lambda e).
\end{align*}
We now note that the sum over triples $(a,b,c)$ is of the same types as in~\eqref{eq:cube quart}
hence Theorem~\ref{thm:Bound T3} applies. 
Recalling that $N=0$ we arrive to the bound
$$
ABCDE  \ll p^{1/8} A^{7/8} B^{29/32}  C^{29/32}  \sum_{\lambda \in \F_p^*} \left|
 \sum_{d \in   \cD }\ep(\lambda d)\right|  \left| \sum_{e \in   \cE} \ep(\lambda e)\right|.
$$
Now applying the Cauchy inequality and then  analogues  of~\eqref{eq:orthog}, we obtain 
\begin{align*}
ABCDE & \ll p^{1/8} A^{7/8} B^{29/32}  C^{29/32}  (pD)^{1/2} (pE)^{1/2}\\
& =  p^{9/8} A^{7/8} B^{29/32}  C^{29/32} D^{1/2} E^{1/2}.
\end{align*}
The asymptotic formula follows.

For the lower bound,  we consider the number $J(\eta)$ of solutions to the equation
$$
(a+b+c)^3+d = \eta, \qquad (a,b,c,d) \in \cA\times \cB \times\cC\times \cD, 
$$
and again proceed as in the proof of Theorem~\ref{thm:Image1} however using 
 Theorem~\ref{thm:Bound T3}  instead of  Theorem~\ref{thm:Bound S3}. 
 In particular, for 
$$
 J =  \sum_{\eta \in \F_p} J(\eta)^2
$$
we obtain 
$$
J \ll \frac{A^2B^2C^2D^2}{p} + p^{1/4} A^{7/4}  B^{29/16}  C^{29/16} D,
$$
from which the desired result follows. 

\subsection{Proof of Theorem~\ref{thm:UV}}

Let $N$ be the number of solutions to the equation
\begin{equation}
\label{eq:Eq bcduv}
u/(bc) + d = v, \qquad (b,c,d,u,v) \in  \cB \times \cC \times  \cD \times \cU \times \cV .
\end{equation}

Clearly the quintuples  $(b,c,d,abc,a+d)$ with 
$(a,b,c,d) \in \cA \times \cB \times \cC \times  \cD$ are pairwise
distinct solutions to~\eqref{eq:Eq bcduv}.
Hence
\begin{equation}
\label{eq:Jlow 2}
N \ge ABCD.
\end{equation}
On the other hand, similarly to~\eqref{eq:N ABCDE}, we  infer from Theorem~\ref{thm:Bound S3} that
\begin{equation}
\label{eq:Jasymp 2}
N =  BCDUV/p +  O\(p^{1/4} U^{3/4} B^{3/4}C^{7/8} D^{1/2} V^{1/2}\) .
\end{equation}
Hence, we see from comparing~\eqref{eq:Jlow 2}  and~\eqref{eq:Jasymp 2}, that  either 
$$
ABCD \ll BCDUV/p
$$
 or 
$$
ABCD \ll  p^{1/4} U^{3/4} B^{3/4}C^{7/8} D^{1/2} V^{1/2}
$$
and we obtain the result.

\section{Comments} 

We now recall a result of  Karatsuba~\cite{Kar1} 
(see also~\cite[Chapter~VIII, Problem~9]{Kar2})
that for the character sums in~\eqref{eq:double char}
we have 
$$
 \max _{\chi \in \Omega^*}  
\left| \sum_{\substack{v \in \cV \\ w \in \cW}} \chi(v-w)\right| \ll 
 V^{1-1/2r}  W q^{1/4r} +  
V^{1-1/2r} W^{1/2} q^{1/2r}, 
$$
where the implied constant depends only on the 
arbitrary parameter $r=1,2, \ldots$. 
Used instead of~\eqref{eq:double char}, this allows us to improve 
Lemma~\ref{lem:N Fq} for some ranges of $U,V,W$, however not 
in ranges which are relevant for our applications to the trilinear  sums~\eqref{eq:Sum S}. 

Furthermore,~\cite{Pet} also contains the bound
\[
D_\times(\cU) \ll \frac{U^8}{p} + p^{2/3} U^{16/3},
\]
which is better than that of Corollary~\ref{cor:DU} when $U \ge p^{4/7}$. It turns out, 
however, that in this range applying the classical bound works better. 

Finally, we note that some of our bounds can be refined and expressed in terms of $\cL^2$-norms
of  some weights instead of their $\cL^\infty$-norms.

\section*{Acknowledgements}

The authors would like to thank Moubariz Garaev  and Misha Rudnev for their 
comments and suggestions.  The authors are also very grateful to 
the referee for the careful reading of the manuscript and 
valuable comments.

During the preparation of this paper, the first  author was supported by the NSF DMS Grant~1500984, 
the second  author was supported by ARC Grant~DP140100118.

\end{document}